\documentclass{article}

\usepackage[cp1251]{inputenc}
\usepackage[T2A]{fontenc}
\usepackage{amssymb}
\usepackage{amsfonts}
\usepackage{amsmath}
\usepackage{amsthm}

\newtheorem{theorem}{Theorem}
\newtheorem{proposition}{Proposition}

\newtheorem{lemma}{Lemma}
\newtheorem{corollary}{Corollary}

\newcommand{\bbR}{{\mathord{\mathbb{R}}}}

\newcommand{\cV}{\mathord{\mathcal{V}}}
\newcommand{\cW}{\mathord{\mathcal{W}}}

\newcommand{\fra}{{\mathord{\mathfrak{a}}}}

\newcommand{\frC}{\mathord{\mathfrak{C}}}

\newcommand{\frh}{\mathord{\mathfrak{h}}}

\newcommand{\frg}{{\mathord{\mathfrak{g}}}}

\newcommand{\frk}{\mathord{\mathfrak{k}}}
\newcommand{\frl}{\mathord{\mathfrak{l}}}

\newcommand{\frt}{\mathord{\mathfrak{t}}}

\newcommand{\frv}{\mathord{\mathfrak{v}}}

\newcommand{\al}{{\mathord{\alpha}}}

\newcommand{\la}{{\mathord{\lambda}}}

\newcommand{\ins}[1]{\quad\mbox{{#1}}\quad}

\newcommand{\imply}{{\ }\Rightarrow{\ }}

\newcounter{sfsf}
\newcommand{\sform}{{\mathop{\mathrm{I}\kern-0.15pt\mathrm{I}}\nolimits}}

\newcommand{\clos}{\mathop{\mathrm{clos}}}

\newcommand{\GL}{\mathop{\mathrm{GL}}}

\newcommand{\LL}{\mathop{\mathrm{L}{}}}

\newcommand{\OO}{\mathop{\mathrm{O}{}}}
\newcommand{\Int}{\mathop{\mathrm{Int}}}
\newcommand{\codim}{\mathop{\mathrm{codim}}\nolimits}

\newcommand{\ad}{\mathop{\mathrm{ad}}}
\newcommand{\Ad}{\mathop{\mathrm{Ad}}}

\newcommand{\scal}[2]{\left<#1,#2\right>}
\newcommand{\spann}{\mathop{\mathrm{span}}}

\newcommand{\reg}{{\mathord{\mathrm{reg}}}}

\newcommand{\SP}{{\sf{SP}}}

\newcommand{\one}{{\mathord{\mathtt1}}}
\let\wh=\widehat

\def\Int{\mathop{\hbox{\rm Int}}\nolimits}

\def\Ad{\mathop{\mbox{\rm Ad}}\nolimits}
\def\cov{\{\wh{O\,}\!_v\}_{v\in\cV}}
\def\co#1{\wh{O\,}\!_{#1}}

\def\clos{\mathop{\hbox{\rm clos}}\nolimits}

\title{\bf  Polar representations of compact groups and
convex hulls of their orbits}
\author{V.\,Gichev}
\date{}
\begin{document}
\maketitle
\begin{abstract}
The paper contains a characterization of compact groups
$G\subseteq\GL(V)$, where $V$ is a finite dimensional real vector
space, which have the following property \SP{}: the family of
convex hulls of $G$-orbits is a semigroup with respect to the
Minkowski addition. If $G$ is finite, then \SP{} holds if and only
if $G$ is a Coxeter group; if $G$ is connected then \SP{} is true
if and only if $G$ is polar. In general, $G$ satisfies \SP{} if
and only if it is polar and its Weyl group is a Coxeter group.
\end{abstract}


\section{Introduction}
A representation of a compact Lie group $G$ in a finite
dimensional Euclidean  vector space $\frv$ is called {\it polar}
if there exists a linear subspace $\fra\subset\frv$ (which is said
to be a {\it Cartan subspace}) such that
\begin{itemize}
\item[(A)] each orbit $O_v=Gv$, where $v\in\frv$, meets
$\fra$;\label{aaa} \item[(B)] for any $u\in\fra$, the tangent
space $\frt_u=T_uO_u$ is orthogonal to $\fra$.\label{bbb}
\end{itemize}
It follows that the set $O_v\cap\fra$ is finite and
$\fra=\frt_u^\bot$ for generic $u\in\fra$. Polar representations
of compact Lie groups were defined and described in \cite{Da}. An
example is the adjoint representation $\Ad$ of $G$ in its Lie
algebra $\frg$, with any Cartan subalgebra as a Cartan subspace. A
more general example is the isotropy representation of a
Riemannian symmetric space (an {\it $s$-representation}). Let
$M=H/K$ be such a space, $\frh=\frk\oplus\frv$ be the Cartan
decomposition, where $\frh,\frk$ are Lie algebras of $H,K$,
respectively, and the space $\frv$ is $\Ad(K)$-invariant. The
representation $\Ad$ of $K$ in $\frv$ is polar since any maximal
abelian subspace $\fra$ of $\frv$ satisfies (A) and (B). By a
result of Dadok (\cite{Da}), a representation of $G$ is polar if
and only if it is {\it orbit equivalent} to some
$s$-representation: there exist $H,K$ as above and an embedding
$G\mapsto K$ such that $G$ is transitive on each $K$-orbit in
$\frv$. In the paper \cite{Da}, the proof involves  a case-by-case
check. A conceptual proof was given in the papers \cite{EH1},
\cite{EH2}. The paper \cite{DK} contains a classification of polar
representations of complex reductive Lie groups (there is a
natural extension of the definition onto this case). The family of
polar representations is not wide (see the papers
\cite{Da}--\cite{EH2} for detailed information). The survey
\cite{VP} contains a comparison of the property to be polar with
other ``good'' properties of representations (for example, to have
a free algebra of invariant polynomials).

In this article, the polar representations of compact Lie groups
are characterized by a semigroup property of their orbits. Let $G$
be a compact subgroup of $\GL(\frv)$; we say that $G$ is {polar}
if its identical representation is polar. For $A,B\subseteq\frv$,
$$A+B=\{a+b:\,a\in A,\,b\in B\}$$
is the Minkowski sum of $A$ and $B$. Let $\wh X$ denote the convex
hull of a set $X\subseteq\frv$.  Here is the semigroup property of
$G$ mentioned above:
\begin{itemize}
\item[\SP:] the family $\{\wh{O\,}\!_v\}_{v\in\frv}$ of convex
hulls of orbits is a semigroup with respect to the Minkowski
addition.
\end{itemize}
The group
\begin{eqnarray}\label{weyl}
W=\{g\in G:\,g\fra=\fra\}\big|_{\fra}
\end{eqnarray}
is said to be the {\it Weyl group} of $G$. Due to (B), it is
finite. In the statement of the following theorem, which is the
main result of the paper,  Coxeter groups are treated as  finite
linear groups generated by reflections in hyperplanes.
\begin{theorem}\label{main}
A compact linear group satisfies \SP{}  if and only if it is polar
and its Weyl group is a Coxeter group.
\end{theorem}
In the definition of a polar representation, it is not assumed
that $G$ is connected. However, this property depends only on the
identity component $G^e$ of $G$; in particular, all finite linear
groups are polar by definition. By the theorem, a finite linear
group satisfies \SP{} if and only if it is a Coxeter group,
including non-crystallographic ones. This is proved in
Theorem~\ref{finv}, which also contains two other geometric
criteria for \SP{} (hence for a finite linear group to be
Coxeter).

Let $G$ be connected. If $G$ is polar, then generic orbits are
isoparametric submanifolds in the ambient Euclidean space (a
submanifold of the Euclidean space is called {\it
isoparametric}{\,} if its normal bundle is  flat and principal
curvatures are constant for any parallel normal vector field). For
isoparametric submanifolds of codimension greater than 2 the
converse is true (i.e., they can be realized as principal orbits
of polar groups, see \cite{Th}; in codimension 2, there are
nonhomogeneous examples). Any compact connected isoparametric
submanifold is naturally associated with a Coxeter group (see, for
example, \cite[Section~6.3]{PT}). If the submanifold is an orbit
of a polar group, then this Coxeter group coincides with the Weyl
group. Thus, if $G$ is connected, then \SP{} holds if and only if
$G$ is polar. The proof of these facts uses the Morse theory; it
would be interesting to know if there exists a direct elementary
proof of \SP{} for connected polar groups.

Let $H$ be a subgroup of $\GL(\frv)$ and $\frC(\frv,H)$ be the
family of all $H$-invariant convex sets in $\frv$. Clearly,
$\frC(\frv,H)$ with the Minkowski addition is a semigroup. The
following proposition is an essential step in the proof of the
theorem.
\begin{proposition}\label{reiso}
If\, $G$ is polar, then the mapping $A\to A\mathop\cap\fra$, where
$\fra$ is a Cartan subspace, is a semigroup isomorphism between
$\frC(\frv,G)$ and $\frC(\fra,W)$.
\end{proposition}

The semigroups of sets in topological groups were considered in
papers \cite{Ra1}, \cite{Ra2}, \cite{BPS}, \cite{BG}. In $\bbR^n$,
one parameter semigroups are of the form $\{tQ\}_{t\geq0}$, where
$Q$ is a convex set.  In particular, the family of closed balls
for any norm is a semigroup. For a left invariant Riemannian
metric in a Lie group, the family of closed balls centered at the
identity is also a semigroup of sets. This property holds for all
left invariant inner metrics; moreover, it characterizes them (see
\cite{BPS}). Semigroups may be parameterized by more general
objects than the numbers. In terms of \cite{BPS}, this defines a
geometry on a group. The orbits of an $s$-representation are
parameterized by points of a closed convex simplicial cone $C$
(the Weyl chamber of the restricted root system). Thus, the
semigroup of their convex hulls can be treated as a vector valued
norm (in general, non-symmetric) on $\frv$ with values in $C$.

The exposition is self-contained and elementary. Some of
preliminary results were published in \cite{Gi1}. We refer to
\cite{Br} and \cite{VP} for general facts on actions of groups.

Throughout the paper, we keep the notation above. Furthermore, let
$f$ be a real function on a set $X$. The set of all $x\in X$ such
that $f(x)=\max_{y\in X}f(y)$ is called the {\it peak set for $f$
on $X$}. A {\it peak point} is a point which is a peak set; in
this case, we say that $f$ {\it has a peak on $X$}.
$\Int_{\frl}(X)$ is the interior of a set $X\subseteq\frl$ in
$\frl\subseteq\frv$. The space $\frv$ is equipped with the inner
product $\scal{\ }{\ }$ such that $G\subseteq\OO(\frv)$. For
$v,u\in\frv$ and compact $X\subseteq\frv$, set
\begin{eqnarray*}
&\fra_v=\frt_v^\bot\subseteq\frv,\\
&\la_v(u)=\scal{v}{u},\\
&\mu_v(X)=H_X(v)=\max_{x\in X}\la_v(x),\\
&P_v(X)=\{x\in X:\,\la_v(x)=\mu_v(X)\}.
\end{eqnarray*}
$H_X$ is the support function for $X$. Note that $u$ is a critical
point for $\la_v$ on $O_u$ if and only if $v\in\fra_u$,
equivalently, if and only if $u\in\fra_v$ ($\frg u\perp{v}$ is the
same as ${u}\perp{\frg v}$). Since $P_v(O_u)$ is the peak set for
$\la_v$ on $O_u$, we get
\begin{eqnarray}\label{peaka}
P_v(O_u)\subseteq\fra_v.
\end{eqnarray}
For all compact $X,Y\subseteq\frv$ and $v\in\frv$, we obviously
have
\begin{eqnarray}
\mu_v(X+Y)=\mu_v(X)+\mu_v(Y),\label{addmu}\\
P_v(X+Y)=P_v(X)+P_v(Y).\label{sumpe}
\end{eqnarray}

The stable subgroup of $v\in\frv$ and its Lie algebra are denoted
by $G_v$, $\frg_v$, respectively. A point $v\in\frv$ is said to be
{\it regular} if $G_v$ is minimal: $G_u\subseteq G_v$ implies
$G_u=G_v$, where $u\in\frv$. For $X\subseteq\frv$, $X^\reg$ is the
set of all regular points in $X$; $\spann(X)$ is the linear span
of $X$. The algebra of linear operators $\frv\to\frv$ is denoted
by $\LL(\frv)$, $e$ is the unit of $G$, $\pi\in\LL(\frv)$ is the
orthogonal projection onto the Cartan subspace $\fra$, and
$\bbR^+=[0,\infty)$.

\section{Preparatory material}
In this section, $G$ is not assumed polar unless this is stated
explicitly. Obviously, the sum of convex sets is convex and the
sum of $G$-invariant sets is $G$-invariant. Hence the inclusion
\begin{equation*}
\co u+\co v\supseteq \co{u+v}
\end{equation*}
holds for all $u,v\in\frv$. The family $\cov$ is a semigroup if
and only if the equality holds for some $u,v$ in every pair of
orbits.

The tangent space $\frt_v=T_vO_v$ may be identified with the
quotient $\frg/\frg_v$ or with the complementary subspace:
\begin{equation*}
\frt_v=\frg v\cong\frg/\frg_v =\frg_v^\bot\subseteq\frg,
\end{equation*}
where $\perp$ relates to some invariant inner product in $\frg$;
then $\frg_v^\bot$ is $\ad(\frg_v)$-invariant.
\begin{lemma}\label{frist}
For any $u\in\frv$, there exists a neighborhood $U$ of $u$ in
$\fra_u$ such that $u$ is a peak point on $O_u$ for all $\la_v$
such that $v\in U$.
\end{lemma}
\begin{proof}
Clearly, $u$ is a peak point for $\la_u$ on $O_u$ and $d^2\la_u$
is negative definite on $\frt_u$. If $v-u$ is small, then the
latter is also true for $d^2\la_v$; if $v\in\fra_u$, then $u$ is a
critical point for $\la_v$. Hence $\la_v$ has a strict local
maximum on $O_u$, which is global if $v\in\fra_u$ is sufficiently
close to $u$.
\end{proof}
Let $C_v^*$ be the closed convex cone hull of the shifted orbit
$O_{-v}$ and $C_v$ be the dual cone to it:
\begin{eqnarray}
C_v&=&\{u\in\frv:\,\scal{u}{v-gv}\geq0\ins{for all}g\in
G\},\label{cvdef}\\
C_v^*&=&\clos\left(\bbR^+(v-\wh O_v)\right).\nonumber
\end{eqnarray}
Clearly, $v\in C_v$. It follows from (\ref{cvdef}) that
\begin{eqnarray}\label{conat}
u\in C_v\quad\Longleftrightarrow\quad v\in P_u(O_v).
\end{eqnarray}
\begin{corollary}\label{cone}
$C_v$ is a closed convex generating cone in $\fra_v$.
\end{corollary}
\begin{proof}
Clearly, $C_v$ is convex and closed.  For all $u,v\in\frv$ we have
\begin{eqnarray}
\mu_u(O_v)=\mu_v(O_u),\label{muvvu}\\
u\in C_v\quad\Longleftrightarrow\quad v\in C_u,\label{equi}\\
u\in P_v(O_u)\quad\Longleftrightarrow\quad v\in
P_u(O_v).\label{peequ}
\end{eqnarray}
Indeed, (\ref{muvvu}) is true since $\max_{g\in G}\scal{u}{gv}=
\max_{g\in G}\scal{v}{g^{-1}u}$, (\ref{equi}) and (\ref{peequ})
follow from (\ref{conat}), (\ref{muvvu}), and the evident equality
$\la_u(v)=\la_v(u)$. By (\ref{conat}), (\ref{peequ}), and
(\ref{peaka}), $C_v\subseteq\fra_v$, moreover,
$\Int_{\fra_v}(C_v)\neq\emptyset$ due to Lemma~\ref{frist}.
\end{proof}
According to Corollary~\ref{cone}, $\frt_v=C_v^\bot$. Hence,
\begin{eqnarray}
C^*_v=(C^*_v\cap\fra_v)+\frt_v\label{pr=pe}.
\end{eqnarray}
Note that $C_v$ is pointed  (and $C_v^*$ is generating) if and
only if $\Int(\wh O_v)\neq\emptyset$. By (\ref{conat}) and
(\ref{peequ}),
\begin{eqnarray}\label{pvou}
P_v(O_u)=C_v\cap O_u.
\end{eqnarray}
\begin{corollary}\label{pere}
Each $G$-orbit in $\frv$ meets $C_v$.
\end{corollary}
\begin{proof} Since $G$ is compact, we have $\la_v(x)=\mu_v(O_u)$
for  some $x\in O_u$.
\end{proof}
Note that each orbit of any connected component of $G$ also meets
$\fra_v$ since there are critical points of $\la_v$ on it;
however, the intersection need not have a common point with $C_v$.

It follows from (\ref{pvou}) that
\begin{eqnarray}\label{ovcv}
O_v\cap C_v=\{v\}.
\end{eqnarray}
\begin{lemma}\label{peak}
Let $u\in C_v$. If $\la_v$ has a peak on $O_u$ at $u$, then
$G_v\subseteq G_u$.
\end{lemma}
\begin{proof}
If $g\in G_v$, then $gu\in C_v$ by (\ref{equi}) since $v=gv\in
gC_u=C_{gu}$. By (\ref{pvou}), $gu\in P_v(O_u)$. Since $u$ is a
peak point for $\la_v$, we get $gu=u$.
\end{proof}
\begin{lemma}\label{sepa}
Let $u\in O_v$ and $u\neq v$. Set $w=u-v$. Then $\left(C_u\cap
C_v\right)\perp w$. Furthermore, $w\not\perp \fra_v$ and the set
$C_u\cap C_v$ is contained in the hyperplane $\fra_v\cap w^\bot$
in $\fra_v$, which is proper. The hyperplane  $w^\bot$ separates
$C_u$ and $C_v$, moreover, $u$ and $v$ are strictly separated.
\end{lemma}
\begin{proof} If $x\in C_u\cap C_v$, then
$\la_x(u)=\la_x(v)=\mu_x(O_v)$ by (\ref{conat}). Hence
$\scal{x}{w}=0$. Further, $w\perp\fra_v$ implies
$|u|^2=|v|^2+|w|^2>|v|^2$ contradictory to $u\in O_v$. For
arbitrary $x\in C_u$ and $y\in C_v$,
\begin{eqnarray*}
\scal{w}{x}=\la_x(u)-\la_x(v)\geq0\geq\la_y(u)-\la_y(v)=\scal{w}{y}.
\end{eqnarray*}
The inequalities hold due to (\ref{conat}). If $x=u$ or $y=v$,
then at least one of the inequalities is strict; if $x=y\in
C_u\cap C_v$, then we have equalities.
\end{proof}
Note that $w^\bot$ is the equidistant hyperplane for $u$ and $v$.
\begin{corollary}\label{empin}
The following assertions hold:
\begin{itemize}
\item[\rm(1)] if $u\in O_v$ and $u\neq v$, then
$\Int_{\fra_v}(C_v)\cap C_u=\emptyset$; \item[\rm(2)] if
$x\in\Int_{\fra_v}(C_v)$, then $C_v\cap O_x=G_vx$.
\end{itemize}
\end{corollary}
\begin{proof}  The hyperplane $w^\bot$ strictly separates
$\Int_{\fra_v}(C_v)$ and $C_u$ due to Lemma~\ref{sepa}. Thus (1)
is true. If $g\in G$, $x\in\Int_{\fra_v}C_v$, and $gx\in C_v$,
then
$$x\in\Int_{\fra_v}\left(C_v\right)\cap C_{g^{-1}v}.$$
By (1), $gv=v$. Conversely, if $gv=v$, then $gC_v=C_v$, hence
$gx\in C_v$. This proves (2).
\end{proof}

We omit the proof of the following lemma, which is standard.
\begin{lemma}\label{krist}
For any $u\in\frv^\reg$ and its neighborhood $V$ in $\frv$, there
exists a neighborhood $U$ of $u$ in $\frv$ with the following
property: if $v\in U$, then  $V\cap O_v$ contains a unique
critical point of $\la_u$ on $O_v$, which is also a peak point for
$\la_u$ on $O_v$. \qed
\end{lemma}

\begin{lemma}\label{nepo}
A vector $v\in\frv$ is regular if and only if $G_v\subseteq G_u$
for all $u\in\fra_v$.
\end{lemma}

\begin{proof}
According to Lemma~\ref{peak}, for all $u$ in the set $U$ of
Lemma~{\rm\ref{frist}} the reverse inclusion $G_u\subseteq G_v$
holds; if $v$ is regular, then $G_u=G_v$. Since $U$ is open in
$\fra_v$, $G_v\subseteq G_u$ for all $u\in\fra_v$. Conversely, let
$G_v\subseteq G_u$ for all $u\in\fra_v$. It follows from
Corollary~\ref{pere} that $\fra_v^\reg\neq\emptyset$. If
$u\in\fra_v^\reg$, then $G_u=G_v$ since $G_u$ is minimal. Thus,
$v$ is regular.
\end{proof} 


We conclude this section with a proposition which combines the
facts on polar groups that we need in the sequel.

\begin{proposition}\label{pol}
Let $G$ be polar, $\fra$ be a Cartan subspace, and $W$ be the Weyl
group. Then $\fra^\reg$ is open and dense in $\fra$ and
$\fra=\fra_v$ for any $v\in\fra^\reg$. Furthermore,
\begin{itemize}
\item[\rm(\romannumeral1)] if $v\in\fra^\reg$, $g\in G$, and
$gv\in\fra$, then $g\fra=\fra$; \item[\rm(\romannumeral2)] if
$v\in\frv^\reg$, then $G_v$ is a normal subgroup of finite index
in the group
\begin{equation*}
G^v=\{g\in G:\,g\fra_v=\fra_v\}=\{g\in G:\,gv\in\fra_v\}
\end{equation*}
and $G^v/G_v\cong W$; \item[\rm(\romannumeral3)] for all
$a\in\fra$, $Wa=O_a\cap\fra$; \item[\rm(\romannumeral4)] for any
$a\in\fra$, $\pi\wh O_a=\wh{Wa}$.
\end{itemize}
\end{proposition}
\begin{proof} It follows from (A) that $\fra^\reg\neq\emptyset$.
Let $v\in\fra^\reg$. By (B), $\frt_v\perp\fra$. Hence
$\fra_v\supseteq\fra$. If $\codim \frt_v>\dim\fra$, then
$\dim(\frv/G)>\dim\fra$, contradictory to (A) and (B). Thus,
$\fra_v=\fra$. It is well known that $\frv^\reg$ is open and dense
in $\frv$. Hence, $\fra^\reg$ is open in $\fra$; it follows from
(A) and (B) that $\fra^\reg$ is dense in $\fra$.

If $v\in\fra^\reg$ and $gv\in\fra$, then $gv\in\fra^\reg$. Hence
$\fra=\fra_{gv}$. Since $g\fra_v=\fra_{gv}$ and $\fra_v=\fra$,
this implies (\romannumeral1).

In (\romannumeral2), we may assume $v\in\fra^\reg$; then
$\fra_v=\fra$. By Lemma~\ref{nepo}, $G_u=G_v$ for all
$u\in\fra^\reg$. Therefore, $G_{gv}=G_v$ if $g\in G^v$, $G_v$ is
normal in $G^v$, and we have $G^v/G_v\cong G^v|_\fra=W$. By
(\romannumeral1), $O_v\cap\fra=G^vv$; thus $Wv=O_v\cap\fra$. It
follows from (B) that $O_v\cap\fra$ and $W$ are finite.

For $a\in\fra^\reg$, (\romannumeral3) was proved above; for all
$a\in\fra$, (\romannumeral3) is true since $\fra^\reg$ is dense in
$\fra$.

In (\romannumeral4), the inclusion $\pi\wh O_a\supseteq\wh{Wa}$ is
obvious. By (\romannumeral3), we have to prove that $\pi\wh
O_a\subseteq\wh{O_a\cap\fra}$. Otherwise, there exist $u\in O_a$
and $b\in\fra^\reg$ such that
\begin{eqnarray*}
\la_b(u)>\max\{\la_b(x):\,x\in O_a\cap\fra\}.
\end{eqnarray*}
Then $P_b(O_u)\cap\fra=\emptyset$ but this contradicts to
(\ref{peaka}) since $\fra=\fra_b$.
\end{proof}

\section{\SP{} for finite linear groups}

If $G$ is finite, then $\fra_v=\frv$ for all $v\in\frv$ and
$\frv^\reg$ consists of $v\in\frv$ such that $G_v=\{e\}$. The
following lemma is a specification of Lemma~\ref{sepa} to this
case.
\begin{lemma}\label{vodir}
For any $v\in\frv$, the family of cones $\{C_{gv}\}_{g\in G}$
defines the Dirichlet---Voronoi partition of $\frv$ for the orbit
$O_v$:
\begin{eqnarray}\label{dirvo}
C_{hv}=\big\{u\in\frv:\,|u-hv|=\min_{g\in G}|u-ghv|\big\}.
\end{eqnarray}
\end{lemma}
\begin{proof}
Let $v\in\frv$. By definition, $-C^*_v$ is the tangent cone to the
convex polytope $\wh O_v$ at the vertex $v$. The dual cone $C_v$
is uniquely determined by the following properties: it contains
$v$ and is bounded by hyperplanes which are orthogonal to extreme
rays of $C_v^*$. Let $\bbR^+(v-gv)$, $g\in G$, be such a ray. Then
the equidistant hyperplane $H$ for $v$ and $gv$ defines a face of
$C_v$ (note that  $|v|=|gv|$ implies $0\in H$). By (\ref{ovcv}),
$O_v\cap C_v=\{v\}$. Hence,  $u\in C_v$ if and only if $v$ is the
nearest point for $u$ in $O_v$. This proves (\ref{dirvo}) for
$h=e$ that is evidently sufficient.
\end{proof}
\begin{corollary}\label{simtr}
If $v\in\frv^\reg$, then the action of $G$ on the family of cones
$C_{gv}$, $g\in\frg$, is simply transitive.\qed
\end{corollary}

\begin{theorem}\label{finv}
Let $G$ be finite. Then \SP{} is equivalent to each of the
following properties:
\begin{itemize}
\item[\rm(\romannumeral1)]  for any $v\in\frv^\reg$, the
functional $\la_v$ has a peak on each $G$-orbit in $\frv$;
\item[\rm(\romannumeral2)] $G$ is a Coxeter group;
\item[\rm(\romannumeral3)] if $v\in\frv^\reg$, then $C_u=C_v$ for
any $u$ from some neighborhood of $v$.
\end{itemize}
\end{theorem}
\rm
\begin{proof}
\SP$\imply$(\romannumeral1). Let $v\in\frv^\reg$, $u,w\in\frv$,
and $\wh O_v+\wh O_u=\wh O_w$. Set $P=P_v(O_u)$. Due to
Corollary~\ref{pere}, we may assume $u,w\in C_v$; then $u\in P$ by
(\ref{pvou}). Clearly, $P_v(\wh O_v)=P_v(O_v)=\{v\}$ and $P_v(\wh
O_u)=\wh P$. By (\ref{sumpe}), we have
\begin{eqnarray*}
P_v(\wh O_w)=v+\wh P.
\end{eqnarray*}
Furthermore, $v+\wh P\subset\Int(C_v)$ since $v\in\Int(C_v)$, $\wh
P\subseteq C_v$, and $\wh P$ is compact. Let $E$ be the set of
extreme points of $v+\wh P$. Clearly, $E\subseteq
v+P\subset\Int(C_v)$. On the other hand, $E\subseteq O_w$ since
extreme points of $P_v(\wh O_w)$ are extreme points of $\wh O_w$.
By Corollary~\ref{empin}, $O_w\cap C_v=G_vw$. The stable subgroup
$G_v$ is trivial since $v\in\frv^\reg$. Therefore, $v+P=\{w\}$ and
$P=\{u\}$.

(\romannumeral1)$\imply$(\romannumeral2). Let $v\in\frv^\reg$,
$u\in O_v$, $u\neq v$. Due to Lemma~\ref{vodir}, we may assume
that $C_v$ and $C_u$ have a common wall
\begin{eqnarray*}
\cW=C_v\cap C_u
\end{eqnarray*}
which is contained in an equidistant hyperplane $H$ such that
$\Int_H\cW\neq\emptyset$. Let $w\in \Int_H(\cW)$. Then $u,v\in
P_w(O_v)$. It follows from (\romannumeral1) that $w$ is not
regular. Hence, $G_w\neq\{e\}$. Let $g\in G_w\setminus\{e\}$. Then
$g C_v\neq C_v$ by Corollary~\ref{simtr}. Obviously, $w$ has a
$G_w$-invariant neighborhood $U$ such that $U\subseteq C_v\cup
C_u$; this implies $g C_v=C_u$. Similar arguments shows that $g
C_u=C_v$. Therefore, $g^2 C_v=C_v$. By Corollary~\ref{simtr},
$g^2=e$ and the condition $g C_v=C_u$ uniquely determines $g$. The
latter means that $g$ is independent of the choice of
$w\in\Int_H(\cW)$. Hence $g$ is identical on some open subset of
$H$. Consequently,  $g$ is a nontrivial involution that fixes
points of $H$. Thus, $g$ is the reflection in $H$. Every pair of
cones in the family $\{C_{gv}\}_{g\in G}$ can be joined by a chain
of these cones in such a way that consecutive ones have a common
wall; since $G$ acts freely on $\{C_{gv}\}_{g\in G}$ by
Corollary~\ref{simtr}, $G$ is generated by reflections.

(\romannumeral2)$\imply$(\romannumeral3),
(\romannumeral2)$\imply$\SP. Let $G$ be a Coxeter group and $C$ be
a Weyl chamber. Then $C$ is a simplicial cone; let
$\varpi_1,\dots,\varpi_n$ be a base in $\frv$ such that
$C=\sum_{k=1}^n\bbR^+\varpi_k$ and let $\al_1,\dots,\al_n$ be the
dual base, which generates the dual cone $C^*$. The group
$G_{\varpi_k}$ is generated by reflections in those walls of $C$
that contain $\varpi_k$ (they correspond to $\al_j$ with $j\neq
k$). Hence, $\wh{G_{\varpi_k}v}\subset\wh O_v$; moreover,
$\wh{G_{\varpi_k}v}$ is a face of $\wh O_v$ which is orthogonal to
$\varpi_k$. This proves inclusions $\wh O_v\cap
C\supseteq(v-C^*)\cap C$ and $C_v\supseteq C$. On the other hand,
the set $\bigcap_{g\in G}g\left(v-C^*\right)$ is convex and
contains $v$. Therefore,
\begin{eqnarray}\label{weyhu}
\wh O_v=\bigcap_{g\in G}g\left(v-C^*\right)=\bigcup_{g\in
G}g\left((v-C^*)\cap C\right)
\end{eqnarray}
and $C_v=C$. This proves (\romannumeral3). Further, \SP{}
evidently holds for the families $\{(v-C^*)\cap C\}_{v\in C}$ and
$\{g(v-C^*)\}_{v\in C}$ for any $g\in G$, hence for $\wh O_v$.

(\romannumeral3)$\imply$(\romannumeral2). Let $C_v=C_u=C$ for $u$
in some open set $U$. Then $C^*_v=C^*_u=C^*$, there $C^*$ is the
dual cone to $C$. Since $v\in\frv^\reg$, we have $G_v=\{e\}$.
Hence, we may assume that $gu\notin hU$ if $g\neq h$, taking a
smaller $U$ if necessary. This implies that each extreme ray of
$C^*$ is of the form $\bbR^+(u-gu)$, where $g\in G$ does not
depend on $u\in U$. Therefore, the linear operator $\one-g$ (where
$\one$ is the identical transformation) maps $U$ into some one
dimensional subspace and is nontrivial. Hence it has rank 1. Since
$g$ is orthogonal, it is a reflection. It remains to note that the
action of $G$ on $O_v$ is simply transitive and that every two
vertices of a convex polytope can be joined by a chain of one
dimensional edges.
\end{proof}

\section{Proof of the main result}

For a polar $G$, $\fra$ is a Cartan subspace, $\pi$ is the
orthogonal projection onto $\fra$, and $W$ is the Weyl group;
$G^e$ denote the identity component of $G$.

\begin{lemma}\label{propo}
Let $G$ be polar. Then
\begin{itemize}
\item[\rm(1)] for any $G$-invariant  convex set $Q\subseteq\frv$,
$\pi Q=Q\cap\fra$ and $Q=G(Q\cap\fra)$; \item[\rm(2)] for every
$W$-invariant convex set $A\subseteq\fra$,  $GA$ is convex and
$\pi GA=A$.
\end{itemize}
\end{lemma}
\begin{proof}
Clearly, $\pi\wh O_v\supseteq\wh O_v\cap\fra\supseteq
\wh{O_v\cap\fra}$. Together with Proposition~\ref{pol},
(\romannumeral3) and (\romannumeral4), this implies $\pi\wh
O_v=\wh O_v\cap\fra$ for all $v\in\frv$.
Since $Q=\bigcup_{v\in Q}\wh O_v$, this proves the first equality
in (1); the second follows from (A).

If $a\in A$, then $\wh{Wa}\subseteq A$ and we get $\pi GA=A$ since
$\pi\wh O_a=\wh{Wa}$ for all $a\in\fra$ by Proposition~\ref{pol},
(\romannumeral4). Thus, $GA\subseteq \pi^{-1}(A)$. Clearly,
$\fra\cap\pi^{-1}(A)=A$ and the set $\pi^{-1}(A)$ is convex. For
any $g\in G$, the same is true for the set $gA$, the space
$g\fra$, and the orthogonal projection $\pi_g$ onto it. Hence
$GA=\bigcap_{g\in G}\pi_g^{-1}(gA)$. This proves that $GA$ is
convex. Thus, (2) is true.
\end{proof}
\begin{proof}[Proof of Proposition~\ref{reiso}]
Clearly, $\pi$ induces a homomorphism
$\frC(\frv,G)\to\frC(\fra,W)$. Since $G$ is polar, $\pi
V=V\cap\fra$ for all $V\in\frC(\frv,G)$ by Lemma~\ref{propo}, (1).
Hence the mapping $\al:\,V\to V\cap\fra$, $V\in\frC(\frv,G)$,
which coincides with $\pi$ on $\frC(\frv,G)$, is a homomorphism.
It follows from Lemma~\ref{propo}, (1), that $\al$ is one-to-one.
By Lemma~\ref{propo}, (2), $GA\in\frC(\frv,G)$ for any
$A\in\frC(\fra,W)$ and $\pi GA=A$. Hence $\al$ is surjective.
\end{proof}
\begin{proof}[Proof of Theorem~\ref{main}]
Let $G$ be polar and $W$ be a Coxeter group. Then the family
$\{\wh{Wa}\}_{a\in\fra}$ is a semigroup due to Theorem~\ref{finv}.
By Proposition~\ref{pol}, ({\romannumeral4}), we have $\wh
O_a=\wh{Wa}$  for all $a\in\fra$. It follows from
Proposition~\ref{reiso} that $\{\wh O_v\}_{v\in\frv}$ is a
semigroup, i.e., $G$ satisfies \SP{}.

Conversely, let \SP{} hold for $G$. Then, for each pair of
$G$-orbits, there exist vectors $u,v$ in them such that $\wh
O_u+\wh O_v=\wh O_{u+v}$. Clearly,
\begin{eqnarray}\label{sumor}
\wh O_u+\wh O_v=\wh O_{u+v}\quad\Longrightarrow\quad
\frt_u+\frt_v\subseteq \frt_{u+v}.
\end{eqnarray}
Suppose $u$ regular. We claim that
\begin{eqnarray*}
\fra=\fra_u
\end{eqnarray*}
is a Cartan subspace. The condition (A) is obvious. By
(\ref{sumor}), $\dim\frt_u\leq\dim\frt_{u+v}$; since
$u\in\fra^\reg$, we have $\dim\frt_u=\dim\frt_{u+v}$. Therefore,
$\frt_u=\frt_{u+v}=\frt_u+\frt_v$. Moreover, $\frt_v=\frt_u$ if
$v$ is regular; then
\begin{eqnarray}\label{vinfra}
\frt_v=\fra^\bot.
\end{eqnarray}
Thus, it is sufficient to prove that there exists a neighborhood
$U$ of $u$ in $\fra$ such that
\begin{eqnarray}\label{peaks}
v\in U\quad\Longrightarrow\quad \wh O_u+\wh O_v=\wh O_{u+v}
\end{eqnarray}
to verify (B). Indeed, we may assume $U\subseteq\fra^\reg$. Then
(\ref{peaks}) implies (\ref{vinfra}) for all $v\in U$;
consequently, $\frt_v\perp\fra$ for all $v\in\fra$.

Let $U$ be such that
\begin{eqnarray}\label{slice}
O_v\cap U=\{v\}
\end{eqnarray}
for each $v\in U$. Since $u\in\frv^\reg$ is a peak point for
$\la_u$ on $O_u$, the function $\la_u$ must have a peak on $O_v$
for $v$ near $u$ by Lemma~\ref{krist}. The peak point $v'\in O_v$
depends on $v$ continuously in some neighborhood of $u$ since $u$
is regular and $d^2\la_u$ is nondegenerate on $\frt_u$.
Furthermore, $v'\in\fra=\fra_u$ since $v'$ is a critical point for
$\la_u$ on $O_v$. Thus, for sufficiently small
 $U$, (\ref{slice}) implies that $v'=v$ if $v\in U$. In other
words, $v$ in (\ref{slice}) is the peak point of $\la_u$ on $O_v$
for all $v\in U$. Further, we have $\wh O_u+\wh O_v=\wh O_w$  for
some $w\in\frv$. Clearly, $\la_u$ have a peak at $u$ on $\wh O_u$
as well as on $O_u$, and the same is true for $\la_u$, $v$, $\wh
O_v$, and $O_v$. Therefore, $\la_u$ has a peak on $\wh O_u+\wh
O_v$ at $u+v$. This implies that $u+v$ is an extreme point for
$\la_u$ on $\wh O_w$ but the set of extreme points of $\wh O_w$
coincides with $O_w$ since $O_w$ is homogeneous. Thus, $u+v\in
O_w$ and we get (\ref{peaks}). This proves the claim.

Since $G$ is polar, we may apply Proposition~\ref{reiso}.
According to it, \SP{} for $G$ implies that the family $\{\wh
O_a\cap\fra\}_{a\in\fra}$ is a semigroup. The convex hull
$\wh{Wa}$ is the least convex $W$-invariant set which contains
$a$, and the same is true for $\wh O_a$ and $G$. The mapping $Q\to
Q\cap\fra$ keeps inclusions and is a bijection by
Proposition~\ref{reiso}. Hence we have $\wh O_a\cap\fra=\wh{Wa}$
for all $a\in\fra$. Therefore, $\{\wh{ Wa}\}_{a\in\fra}$ is a
semigroup; by Theorem~\ref{finv}, $W$ is a Coxeter group.
\end{proof}

\vbox{\vskip1cm
\noindent V.M. Gichev\\
gichev@ofim.oscsbras.ru\\
Omsk Branch of Sobolev Institute of Mathematics\\
Pevtsova, 13, 644099\\
Omsk, Russia}


\begin{thebibliography}{9999}\rm
\bibitem{BG}
Berestovskii, V.N., Gichev, V.M., \it Metrized left invariant
orders on topological groups \rm (English. Russian original), St.
Petersburg. Math. J. 11, No. 4, 543-566 (2000); translation from
Algebra Anal. 11, No. 4, 1-34 (1999).
\bibitem{BPS} Berestovski\v\i{} V., Plaut C., Stallmann C., \it
Geometric groups I, \rm Trans. AMS, 351 (1999), 1403-1422.
\bibitem{Br}
Bredon G.E., {\it Introduction to compact transformation groups,
Academic Press}, New York, 1972
\bibitem{Da}
Dadok J., \it Polar coordinates induced by actions of compact Lie groups,
\rm Trans. of the AMS, v. 288 (1985), no. 1, p. 125--137.
\bibitem{DK}
Dadok J., Kac V., \it Polar representations, \rm J. of Algebra, v.
92 (1985), p. 504--524.
\bibitem{EH1}
Eschenburg J.-H., Heintze E., \it Polar Representations and Symmetric
Spaces, \rm J. reine angew. Math., v. 507 (1999), p. 93--106.
\bibitem{EH2}
Eschenburg J.-H., Heintze E., \it On the classification of polar
representations, \rm Math. Z., v. 232 (1999), p. 391--398.
\bibitem{Gi1}
Gichev, V.M., \it A remark on polar representations of compact Lie
groups.\rm (Russian). Mat. Strukt. Model. 6, 29-35 (2000).
\bibitem{PT}
Palais, R., Terng, C.-L., \it Critical point theory and
submanifold geometry, \rm Lecture Notes in Mathematics, Vol. 1353,
Springer-Verlag, Berlin, Heidelberg, New York, 1988.
\bibitem{Ra1} Radstr\"om H., \it Convexity and norm in topological
         groups, \rm Arkiv f\"ur Ma\-te\-ma\-tik, \bf 2 \rm (1952),
         \char194\relax\,7, 99--137.
\bibitem{Ra2} Radstr\"om H., \it One-parameter semigroups of subsets
         of a real linear space, \rm Arkiv f\"ur Ma\-te\-ma\-tik,
             \bf 4 \rm (1959), \char194\relax\,9, 87--97.
\bibitem{Th} Thorbergsson G., \it Isoparametric foliations and
their buildings, \rm Annals of Math., v.133, no. 2 (1991),  pp.
429--446
\bibitem{VP}
Vinberg E.B., Popov V.L., \it Invariant theory, \rm (English.
Russian original) Algebraic geometry IV, Encyclopedia of Math.
Sci., 55, Springer-Verlag, 1994, 123–278; translation from Itogi
Nauki i Tekhniki. Ser. Sovrem. Probl. Mat. Fund. Napr., 55,
VINITI, Moscow, 1989, 137–309.
\end{thebibliography}
\end{document}